\documentclass[11pt]{article}
\usepackage{amsmath}
\usepackage{dsfont}
\usepackage{mathrsfs}
\usepackage{amsmath,amssymb}
\usepackage{amsfonts}
\usepackage{hyperref}
\usepackage{amsthm}
\usepackage{graphicx}
\usepackage{subfigure}
\usepackage{xcolor}
\renewcommand{\qed}{\hfill\small{$\square$}\normalsize}

\hfuzz=\maxdimen
\theoremstyle{definition}
\newtheorem{lemma}{Lemma}[section]

\newtheorem{theorem}[lemma]{Theorem}

\newtheorem{remark}{Remark}

\numberwithin{equation}{section}

\renewcommand{\qed}{\hfill\small{$\square$}\normalsize}

\DeclareFixedFont{\Acknowledgment}{OT1}{cmr}{bx}{n}{14pt}
\begin{document}

\title{\bf $p$-th Kazdan-Warner equation on graph}
\author{Huabin Ge}
\maketitle

\begin{abstract}
Let $G=(V,E)$ be a connected finite graph and $C(V)$ be the set of functions defined on $V$. Let $\Delta_p$ be the discrete $p$-Laplacian on $G$ with $p>1$ and $L=\Delta_p-k$, where $k\in C(V)$ is positive everywhere. Consider the operator $L:C(V)\rightarrow C(V)$. We prove that $-L$ is one to one, onto and preserves order. So it implies that there exists a unique solution to the equation $Lu=f$ for any given $f\in C(V)$. We also prove that the equation $\Delta_pu=\overline{f}-f$ has a solution which is unique up to a constant, where $\overline{f}$ is the average of $f$. With the help of these results, we finally give various conditions such that the $p$-th Kazdan-Warner equation
$\Delta_pu=c-he^u$ has a solution on $V$ for given $h\in C(V)$ and $c\in \mathds{R}$. Thus we generalize Grigor'yan, Lin and Yang's work \cite{GLY} for $p=2$ to any $p>1$.

\end{abstract}


\section{Introduction}\label{Introduction}
Let $\widetilde{g}=e^{2\varphi}g$ be a conformal deformation of a smooth metric $g$ on a $2$-dimensional compact Riemannian manifold $M$ without boundary. To find a smooth function $\widetilde{K}$ as the Gaussian curvature of $\widetilde{g}$, one needs to solving the nonlinear elliptic equation
\begin{equation}\label{equ-smooth-original}
\Delta_g\varphi=K-\widetilde{K}e^{2\varphi}.
\end{equation}
By a parameter transformation of $\varphi$ to $u$, the equation (\ref{equ-smooth-original}) takes the following form
\begin{equation}\label{equ-smooth-simplify}
\Delta_gu=c-he^{u},
\end{equation}
where $c$ is a constant, and $h$ is some prescribed function, with neither $c$ nor $h$ depends on the geometry of $(M,g)$. Kazdan and Warner \cite{KW} had given satisfying characterizations to the solvability of the equation (\ref{equ-smooth-simplify}).

Grigor'yan, Lin and Yang \cite{GLY} first studied the corresponding finite graph version of the Kazdan-Warner equation (\ref{equ-smooth-simplify}), that is, $\Delta u=c-he^{u}$, where $\Delta$ is a discrete graph Laplacian, $c\in\mathds{R}$, and $h$ is a function defined on the vertices.

The smooth Laplace-Beltrami operator $\Delta:C^{\infty}\rightarrow C^{\infty}$ is defined as $\Delta f=-\mathrm{div}(\nabla f)$. Those $f$ with $\Delta f=0$ (called harmonic function) are critical points of the Dirichlet energy $\int_M|\nabla f|^2dv_g$. For any $p>1$, one can define a smooth $p$-Laplacian $\Delta_p:C^{\infty}\rightarrow C^{\infty}$ as
$$\Delta_pf=-\mathrm{div}(|\nabla f|^{p-2}\nabla f).$$
Those $f$ with $\Delta_p f=0$ (often called $p$-harmonic function) are critical points of the $p$-th Dirichlet energy $\int_M|\nabla f|^pdv_g$. Similarly, one can define a $p$-th discrete Laplacian $\Delta_p$ on a graph. In this paper, we study the following equation on a finite graph
\begin{equation*}
\Delta_p u=c-he^{u},
\end{equation*}
which is called the $p$-th Kazdan-Warner equation on graph in the paper.

We shall generalize Grigor'yan, Lin and Yang's results \cite{GLY} for $p=2$ to any $p>1$. The difficulties for the generalizations mainly come from the fact that $\Delta_p$ is a nonlinear operator when $p\neq2$. We first prove some very important theorems (Theorem \ref{thm-L}-\ref{thm-bar-f-equation}) related to the discrete $p$-Laplacian $\Delta_p$. With the help of these theorems, we give satisfying characterizations (Theorem \ref{thm-c=0}-\ref{thm-c<0-h<0}) to the solvability of the $p$-th Kazdan-Warner equation on a finite graph. The main theorems in the paper are proved using variational principles and the method of upper and lower solutions. The main idea of the paper comes from reading of Grigor'yan, Lin and Yang \cite{GLY}, Kazdan and Warner \cite{KW}. We follow the approach pioneered by Kazdan and Warner \cite{KW}.

The paper is organized as follows. In section \ref{sect-main-result}, we state the main theorems of the paper. In section \ref{sect-preliminary-lemma}, we give some useful lemmas, namely, a Liouville-type theorem, the Sobolev embedding and the Trudinger-Moser embedding. In section \ref{sect-prof-thm-L}-\ref{sect-prof-thm-c<0-h<0}, we prove Theorem \ref{thm-L}-\ref{thm-c<0-h<0} respectively.

\section{Settings and main results}
Let $G=(V,E)$ be a finite graph, where $V$ denotes the vertex set and $E$ denotes the edge set. Fix a vertex measure $\mu:V\rightarrow(0,+\infty)$ and an edge measure $\omega:E\rightarrow(0,+\infty)$ on $G$. The edge measure $\omega$ is assumed to be symmetric, that is, $\omega_{ij}=\omega_{ji}$ for each edge $i\thicksim j$.

Denote $C(V)$ as the set of all real functions defined on $V$, then $C(V)$ is a finite dimensional linear space with the usual function additions and scalar multiplications. For any $p>1$, the $p$-th discrete graph Laplacian $\Delta_p:C(V)\rightarrow C(V)$ is
\begin{equation*}
\Delta_pf_i=\frac{1}{\mu_i}\sum\limits_{j\thicksim i}\omega_{ij}|f_j-f_i|^{p-2}(f_j-f_i)
\end{equation*}
for any $f\in C(V)$ and $i\in V$. If $p\neq2$, then $\Delta_p(f+g)\neq\Delta_pf+\Delta_pg$ for general $f,g\in C(V)$. For $\lambda\in \mathds{R}$, let $\mathrm{sgn}(\lambda)=1$, $0$ or $-1$ if $\lambda>0$, $=0$ or $<0$, then $\Delta_p(\lambda f)=\mathrm{sgn}(\lambda)|\lambda|^{p-1}\Delta_pf$.
For any $f\in C(V)$, denote $\overline{f}$ as the average value of $f$ with respect to the vertex measure $\mu$. Throughout this paper, all graphs are assumed to be connected.

The following two results are useful in the study of the $p$-th Kazdan-Warner equations on finite graphs. We state them as theorems since they are also important as themselves.

\label{sect-main-result}
\begin{theorem}\label{thm-L}
Let $G=(V,E)$ be a finite graph, and $k\in C(V)$ is positive everywhere. Consider the operator $L=\Delta_p-k:C(V)\rightarrow C(V)$. Then $L$ is one to one and onto. Moreover, $-L^{-1}$ preserves order, that is, if $-Lf\leq-Lg$, then $f\leq g$.
\end{theorem}

\begin{theorem}\label{thm-bar-f-equation}
Let $G=(V,E)$ be a finite graph. For any given $f\in C(V)$, there exists a solution to the equation $\Delta_pu=\overline{f}-f$. Moreover, if
$u'$ is also a solution, then $u'$ differs from $u$ by a constant.
\end{theorem}

Rewrite the $p$-th Kazdan-Warner equation on the graph $G$ as follows:
\begin{equation}\label{def-p-KW-eq}
\Delta_pu=c-he^u \;\;\;\mathrm{in}\; V,
\end{equation}
where $c\in \mathds{R}$ is a constant, and $h\in C(V)$ is a function. We summarize various conditions such that (\ref{def-p-KW-eq}) has a solution as the following four theorems.\\

In case $c=0$, we have:
\begin{theorem}\label{thm-c=0}
Let $G=(V,E)$ be a finite graph. Consider the equation (\ref{def-p-KW-eq}) with $c=0$. Assume $h\not\equiv0$, then (\ref{def-p-KW-eq}) has a solution if and only if $h$ changes sign and $\overline{h}<0$.
\end{theorem}

\begin{remark}
If $h\equiv0$, then the equation (\ref{def-p-KW-eq}) changes to $\Delta_pu=c$. By Lemma \ref{lem-Liouville}, we know the equation $\Delta_pu=c$ have solutions if only if $c=0$. In this case, the only solutions are constant functions.
\end{remark}

In case $c>0$, we have:

\begin{theorem}\label{thm-c>0}
Let $G=(V,E)$ be a finite graph. Consider the equation (\ref{def-p-KW-eq}) with $c>0$. Then (\ref{def-p-KW-eq}) has a solution if and only if $h$ is positive somewhere.
\end{theorem}

In case $c<0$, we have:
\begin{theorem}\label{thm-c<0}
Let $G=(V,E)$ be a finite graph. Consider the equation (\ref{def-p-KW-eq}) with $c<0$. If (\ref{def-p-KW-eq}) has a solution, then $\overline{h}<0$. On the contrary, if $\overline{h}<0$, then there exists a constant $c_-(h)\in[-\infty,0)$ depending only on $h$ such that (\ref{def-p-KW-eq}) has a solution for any $c\in(c_-(h),0)$, but has no solution for any $c<c_-(h)$.
\end{theorem}
If $h$ is non-positive and not identically zero, we can show more:
\begin{theorem}\label{thm-c<0-h<0}
Let $G=(V,E)$ be a finite graph. Consider the equation (\ref{def-p-KW-eq}) with $c<0$. Assume $h\leq0$ and $h\not\equiv0$, then $c_-(h)=-\infty$ and hence (\ref{def-p-KW-eq}) always has a solution, where $c_-(h)$ is given as in Theorem \ref{thm-c<0}.
\end{theorem}

\section{Preliminaries and some useful lemmas}
\label{sect-preliminary-lemma}
Throughout this paper, denote $q$ as the conjugate number of $p$, that is, $\frac{1}{p}+\frac{1}{q}=1$. For any $f\in C(V)$, denote $f_M$ (or $f_m$) as the maximum (or minimum) value of $f$, and denote $C_{p,f,G}$ as some positive constant depending only on the information of $f$, $p$ and $G$. Note that the information of $G$ contains $V$, $E$, $\mu$ and $\omega$. The following lemma can be considered as a Liouville-type theorem on a finite graph:
\begin{lemma}\label{lem-Liouville}
Given $f\in C(V)$, if $\Delta_pf\geq0$ (or $\leq 0$), then $f$ equals to a constant.
\end{lemma}
\begin{proof}
Suppose $\Delta_pf\geq0$. Choose $i\in V$, such that $f_i=\max\limits_{j\in V}f_j$. Then
$$0\leq\Delta_pf_i=\frac{1}{\mu_i}\sum\limits_{j\thicksim i}\omega_{ij}|f_j-f_i|^p(f_j-f_i)\leq0.$$
Hence $f_j=f_i$ for all $j\thicksim i$. Since $G$ is connected, $f$ equals to a constant. If $\Delta_pf\leq0$, then $\Delta_p(-f)\geq0$,
and hence $f$ also equals to a constant.\qed\\
\end{proof}

For any $f\in C(V)$, we define an integral of $f$ over $V$ with respect to the vertex weight $\mu$ by
$$\int_Vfd\mu=\sum\limits_{i\in V}\mu_if_i.$$
Thus $\overline{f}=\int_Vfd\mu\big/\int_Vd\mu$. Set $\mathrm{Vol}(G)=\int_Vd\mu$. Similarly, for any function $g$ defined on the edge set $E$, we define an integral of $g$ over $E$ with respect to the edge weight $\omega$ by
$$\int_Egd\omega=\sum\limits_{i\thicksim j}\omega_{ij}g_{ij}.$$
Specially, for any $f\in C(V)$,
$$\int_E|\nabla f|^pd\omega=\sum\limits_{i\thicksim j}\omega_{ij}|f_j-f_i|^p,$$
where $|\nabla f|$ is defined on the edge set $E$, and $|\nabla f|_{ij}=|f_j-f_i|$ for each edge $i\thicksim j$. Next we consider the Sobolev space $W^{1,\,p}$ on the graph $G$. Define
$$W^{1,\,p}(G)=\{u\in C(V):\int_E|\nabla\varphi|^pd\omega+\int_V|\varphi|^pd\mu<+\infty\},$$
and
$$\|u\|_{W^{1,\,p}(G)}=\left(\int_E|\nabla\varphi|^pd\omega+\int_V|\varphi|^pd\mu\right)^{\frac{1}{p}}.$$
Since $G$ is a finite graph, then
$W^{1,\,p}(G)$ is exactly $C(V)$, a finite dimensional linear space. This implies the following Sobolev embedding:

\begin{lemma}\label{lem-Sobolev-embedding}
Let $G=(V,E)$ be a finite graph. The Sobolev space $W^{1,\,p}(G)$ is pre-compact. Namely, if $\{\varphi_n\}$ is bounded in $W^{1,\,p}(G)$, then there exists some $\varphi\in W^{1,\,p}(G)$ such that up to a subsequence, $\varphi_n\rightarrow\varphi$ in $W^{1,\,p}(G)$.
\end{lemma}
\begin{remark}
The convergence in $W^{1,\,p}(G)$ is in fact pointwise convergence.
\end{remark}
As a consequence of Lemma \ref{lem-Sobolev-embedding}, we have the following $p$-th Poincar\'{e} inequality:
\begin{lemma}\label{lem-p-poincare}
Let $G=(V,E)$ be a finite graph. For all functions $\varphi\in C(V)$ with $\overline{\varphi}=0$, there exists some positive constant $C_{p,G}$ depending on $G$ and $p$ such that
$$\int_V|\varphi|^pd\mu\leq C_{p,G}\int_E|\nabla\varphi|^pd\omega.$$
\end{lemma}
\begin{proof}
If the conclusion is not true, we can choose a sequence of functions $\{\varphi_n\}$ satisfying $\overline{\varphi_n}=0$, $\int_V|\varphi_n|^pd\mu=1$, but $\int_E|\nabla\varphi_n|^pd\omega\rightarrow0$ as $n\rightarrow+\infty$. Thus $\|\varphi_n\|_{W^{1,\,p}(G)}$ is bounded. By Lemma \ref{lem-Sobolev-embedding}, there exists some $\varphi\in C(V)$ such that up to a subsequence, $\varphi_n\rightarrow\varphi$
in $W^{1,\,p}(G)$. We may well denote this subsequence as $\varphi_n$. Hence
$$\int_E|\nabla\varphi|^pd\omega=\lim\limits_{n\rightarrow\infty}\int_E|\nabla\varphi_n|^pd\omega=0.$$
This implies that $\varphi$ is a constant and hence
$$\varphi\equiv\overline{\varphi}=\lim\limits_{n\rightarrow\infty}\overline{\varphi_n}=0,$$
which contradicts with
$$\int_V|\varphi|^pd\mu=\lim\limits_{n\rightarrow\infty}\int_V|\varphi_n|^pd\mu=1.$$
 \qed
\end{proof}

\begin{lemma}\label{lem-Trudinger-Moser}
Let $G=(V,E)$ be a finite graph. Fix any $\beta\in\mathds{R}$ and any $\alpha\geq0$. Then there exist a constant $C_{p,G}$ and a constant $C_{\alpha,\beta,p,G}$, such that for all functions $\varphi\in C(V)$ with $\int_E|\nabla\varphi|^pd\omega\leq1$ and $\overline{\varphi}=0$, there holds
$$\|\varphi\|_{\infty}\leq C_{p,G}$$
and
$$\int_Ve^{\beta|\varphi|^{\alpha}}d\mu\leq C_{\alpha,\beta,p,G}.$$
\end{lemma}
\begin{proof}
By the H\"{o}lder inequality and Lemma \ref{lem-p-poincare}, we get
\begin{equation*}
\begin{aligned}
\mu_{m}\|\varphi\|_{\infty}\leq\mu_{m}\sum\limits_{i\in V}|\varphi|_i\leq & \int_V|\varphi|d\mu\\
\leq& \left(\int_V1^qd\mu\right)^{\frac{1}{q}}\left(\int_V|\varphi|^pd\mu\right)^{\frac{1}{p}}\\
\leq& C_{p,G}\left(\int_E|\nabla\varphi|^pd\omega\right)^\frac{1}{p}\leq C_{p,G}.
\end{aligned}
\end{equation*}
Hence $\|\varphi\|_{\infty}\leq \mu_{m}^{-1}C_{p,G}=C_{p,G}$. If $\beta\leq0$, then obviously
$$\int_Ve^{\beta|\varphi|^{\alpha}}d\mu\leq \int_V1d\mu\leq C_{G}.$$
If $\beta\geq0$, then
$$\int_Ve^{\beta|\varphi|^{\alpha}}d\mu\leq \int_Ve^{\beta\|\varphi\|_{\infty}^{\alpha}}d\mu\leq
\int_Ve^{\beta C_{p,G}^{\alpha}}d\mu\leq C_{\alpha,\beta,p,G}.$$
\end{proof}

\section{Proof of Theorem \ref{thm-L}}
\label{sect-prof-thm-L}
\subsection{$L$ is one to one}\label{subsect-L-one-to-one}
Suppose $(\Delta_p-k)f=(\Delta_p-k)g$, we need to show $f=g$. For each $i\in V$, we have
$$\Delta_pf_i-\Delta_pg_i=k_i(f_i-g_i).$$
Hence
$$\sum\limits_{j\thicksim i}\omega_{ij}\left(|f_j-f_i|^{p-2}(f_j-f_i)-|g_j-g_i|^{p-2}(g_j-g_i)\right)=\mu_ik_i(f_i-g_i).$$
For any $u,v\in C(V)$, denote
$$A(u,v)=\sum\limits_{i\thicksim j}\omega_{ij}|u_j-u_i|^{p-2}(u_j-u_i)(v_j-v_i),$$
then for the above $f$ and $g$, we have
\begin{equation*}
\begin{aligned}
|A(f,g)|\leq&\sum\limits_{i\thicksim j}\omega_{ij}|f_j-f_i|^{p-1}|g_j-g_i|\\
=&\sum\limits_{i\thicksim j}\big(\omega_{ij}^{1/p}|g_j-g_i|\big)\big(\omega_{ij}^{1/q}|f_j-f_i|^{p-1}\big)\\
\leq & \Big(\sum\limits_{i\thicksim j}\omega_{ij}|g_j-g_i|^p\Big)^{\frac{1}{p}}
\Big(\sum\limits_{i\thicksim j}\omega_{ij}|f_j-f_i|^{(p-1)q}\Big)^{\frac{1}{q}}\\
\leq&\frac{1}{p}\sum\limits_{i\thicksim j}\omega_{ij}|g_j-g_i|^{p}+\frac{1}{q}\sum\limits_{i\thicksim j}\omega_{ij}|f_j-f_i|^{p}\\
=&\frac{1}{q}\int_E|\nabla f|^pd\omega+\frac{1}{p}\int_E|\nabla g|^pd\omega.
\end{aligned}
\end{equation*}
Note we have used the inequality $a^{\frac{1}{p}}b^{\frac{1}{q}}\leq\frac{a}{p}+\frac{b}{q}$ for any $a,b\geq0$. Similarly, we can prove
$$|A(g,f)|\leq\frac{1}{q}\int_E|\nabla g|^pd\omega+\frac{1}{p}\int_E|\nabla f|^pd\omega.$$
This leads to
\begin{equation*}
\begin{aligned}
0\leq&\sum\limits_{i\in V}\mu_ik_i(f_i-g_i)^2\\
=&\sum\limits_{i\in V}(f_i-g_i)\sum\limits_{j\thicksim i}\omega_{ij}\left(|f_j-f_i|^{p-2}(f_j-f_i)-|g_j-g_i|^{p-2}(g_j-g_i)\right)\\
=&\sum\limits_{j\thicksim i}\omega_{ij}\big((f_i-g_i)-(f_j-g_j)\big)\left(|f_j-f_i|^{p-2}(f_j-f_i)-|g_j-g_i|^{p-2}(g_j-g_i)\right)\\
=&\sum\limits_{i\thicksim j}\omega_{ij}|f_j-f_i|^{p-2}(f_j-f_i)(g_j-g_i)+\sum\limits_{i\thicksim j}\omega_{ij}|g_j-g_i|^{p-2}(g_j-g_i)(f_j-f_i)\\
&-\Big(\sum\limits_{i\thicksim j}\omega_{ij}|f_j-f_i|^p+\sum\limits_{i\thicksim j}\omega_{ij}|g_j-g_i|^p\Big)\\
=&A(f,g)+A(g,f)-\left(\int_E|\nabla f|^pd\omega+\int_E|\nabla f|^pd\omega\right)\leq0,
\end{aligned}
\end{equation*}
which implies $f=g$.

\subsection{$L$ is on to}
For any given $f\in C(V)$, we need to show the equation $\Delta_pu-ku=f$ has a solution $u$. Denote
$$E(\varphi)=\frac{1}{p}\int_E|\nabla \varphi|^pd\omega+\frac{1}{2}\int_Vk\varphi^2d\mu+\int_Vf\varphi d\mu, \;\;\varphi\in C(V).$$
Consider the Euler-Lagrange equation of $E(\varphi)$. By calculation, we get
\begin{equation*}
\frac{d}{dt}\Big|_{t=0}E(\varphi+t\phi)=-\int_V(\Delta_p\varphi-k\varphi-f)\phi d\mu.
\end{equation*}
Hence $\nabla E(\varphi)=0$ if and only if $\Delta_p\varphi-k\varphi=f$. Next we show $E(\varphi)\rightarrow+\infty$ as $\|\varphi\|\rightarrow+\infty$, and hence $E(\varphi)$ attains its minimum in $C(V)$, which is a finite dimensional linear space.
Since $$\left|\int_Vf\varphi d\mu\right|\leq\|f\|_2\|\varphi\|_2=C_{f,G}\|\varphi\|_2,$$
then
\begin{equation*}
\begin{aligned}
E(\varphi)\geq& \frac{1}{p}\int_E|\nabla \varphi|^pd\omega+\frac{1}{2}\int_Vk\varphi^2d\mu-C_{f,G}\|\varphi\|_2\\
\geq& \frac{1}{2}\int_Vk\varphi^2d\mu-C_{f,G}\|\varphi\|_2\\
\geq&\frac{k_m}{2}\big(\|\varphi\|_2^2-C_{f,G}\|\varphi\|_2\big)\rightarrow+\infty
\end{aligned}
\end{equation*}
as $\|\varphi\|\rightarrow+\infty$. Suppose $E(\varphi)$ attains its minimum at $u\in C(V)$, then $\nabla E(u)=0$, which implies
$\Delta_pu-ku=f$.
\subsection{$-L^{-1}$ preserves order}
We first give a lemma, the proof of which is elementary.
\begin{lemma}\label{lem-a<b}
Assume $a,b\in\mathds{R}$, $a\leq b$ and $p>1$, then $|a|^{p-2}a\leq|b|^{p-2}b$.
\end{lemma}
\begin{proof}
We can easily see that the conclusion above is valid under three complete cases, that is, $a\leq b\leq 0$, $a\leq0\leq b$ and $0\leq a\leq b$. We leave out the details.
\qed\\
\end{proof}

Next suppose $Lf\geq Lg$, we need to prove $f\leq g$. If not, then there is a vertex $i\in V$ such that
$$f_i-g_i=\max\limits_{j\in V}(f_j-g_j)>0.$$
Hence $f_i-g_i\geq f_j-g_j$ and then $f_j-f_i\leq g_j-g_i$. Then by Lemma \ref{lem-a<b}, we get
$$|f_j-f_i|^{p-2}(f_j-f_i)\leq|g_j-g_i|^{p-2}(g_j-g_i).$$
Since $Lf_i\geq Lg_i$, or say $\Delta_pf_i-k_if_i\geq\Delta_pg_i-k_ig_i$, we have
\begin{equation*}
\begin{aligned}
0<\mu_ik_i(f_i-g_i)\leq&\Delta_pf_i-\Delta_pg_i\\
=&\sum\limits_{j\thicksim i}\omega_{ij}\big(|f_j-f_i|^{p-2}(f_j-f_i)-|g_j-g_i|^{p-2}(g_j-g_i)\big)\leq0,
\end{aligned}
\end{equation*}
which is a contradiction. Thus $-Lf\leq-Lg$ implies $f\leq g$, and $-L$ preserves order.

\section{Proof of Theorem \ref{thm-bar-f-equation}}
\label{sect-prof-thm-bar-f-equation}
We can prove the uniqueness (up to a constant) of the equation $\Delta_pu=\overline{f}-f$ following similar process in subsection \ref{subsect-L-one-to-one}. The details are some what tedious and are omitted here. We mainly concentrate to the solvability of the equation $\Delta_pu=\overline{f}-f$ for any given $f\in C(V)$ in this subsection. If $f$ equals to a constant, then any constant function is a solution. Hence in the following we suppose $f$ is not a constant. Define an energy functional $F(\varphi)$ for any $\varphi\in C(V)$ as
$$F(\varphi)=\frac{1}{p}\int_E|\nabla \varphi|^pd\omega+\int_V(\overline{f}-f)\varphi d\mu.$$
Consider the Euler-Lagrange equation of $F(\varphi)$. For any $\phi\in C(V)$, by calculation, we get
\begin{equation*}
\frac{d}{dt}\Big|_{t=0}F(\varphi+t\phi)=\int_V(-\Delta_p\varphi+\overline{f}-f)\phi d\mu.
\end{equation*}
Hence $\nabla F(\varphi)=0$ if and only if $\Delta_p\varphi=\overline{f}-f$. Next we estimate $F(\varphi)$. By Lemma \ref{lem-p-poincare},
\begin{equation*}
\begin{aligned}
\left|\int_V(\overline{f}-f)(\varphi-\overline{\varphi})d\mu\right|\leq&
\left(\int_V|\overline{f}-f|^qd\mu\right)^{\frac{1}{q}}\left(\int_V|\varphi-\overline{\varphi}|^pd\mu\right)^{\frac{1}{p}}.\\
\leq&C_{p,f,G}\left(C_{p,G}\int_E|\nabla(\varphi-\overline{\varphi})|^pd\omega\right)^{\frac{1}{p}}\\
=&C_{p,f,G}\left(\int_E|\nabla\varphi|^pd\omega\right)^{\frac{1}{p}}.
\end{aligned}
\end{equation*}
Then it follows
\begin{equation*}
\begin{aligned}
F(\varphi)=&\frac{1}{p}\int_E|\nabla\varphi|^pd\omega+\int_V(\overline{f}-f)(\varphi-\overline{\varphi})d\mu+
\int_V(\overline{f}-f)\overline{\varphi}d\mu\\
=&\frac{1}{p}\int_E|\nabla\varphi|^pd\omega+\int_V(\overline{f}-f)(\varphi-\overline{\varphi})d\mu\\
\geq&\frac{1}{p}\int_E|\nabla\varphi|^pd\omega-C_{p,f,G}\left(\int_E|\nabla\varphi|^pd\omega\right)^\frac{1}{p}.
\end{aligned}
\end{equation*}
Since $p>1$, then for any positivie constant $a$, $x-ax^{\frac{1}{p}}>\frac{1}{2}x$ if $x>0$ is big enough (such as, $x>(2a)^{\frac{p}{p-1}}$). Using this fact and Lemma \ref{lem-p-poincare}, we know
$$F(\varphi)\geq\frac{1}{2p}\int_E|\nabla\varphi|^pd\omega\geq C_{p,G}\int_V|\varphi-\overline{\varphi}|^pd\mu$$
if $\int_V|\varphi-\overline{\varphi}|^pd\mu$ is big enough (such as, bigger than some constant $C_{p,f,g}$).

Look at the linear space $C(V)$ with inner product $\langle \varphi,\phi\rangle=\int_V\varphi\phi d\mu$. Let $\Lambda$ be the hyperplane orthogonal to the constant function $1\in C(V)$. Obviously $\overline{\varphi}=0$ for any $\varphi\in\Lambda$. Moreover, $\varphi-\overline{\varphi}\in\Lambda$ for any $\varphi\in C(V)$. Thus if $\varphi\in\Lambda$ and $\|\varphi\|\rightarrow+\infty$, then
$$\int_V|\varphi-\overline{\varphi}|^pd\mu=\int_V|\varphi|^pd\mu\rightarrow+\infty$$
and hence
$F(\varphi)\rightarrow+\infty$. This implies that $F|_{\Lambda}$ attains its minimum at some $u\in\Lambda$. Recall the Euler-Lagrange equation of $F$, we know for any $\phi\in\Lambda$:
$$\int_V(-\Delta_pu+\overline{f}-f)\phi d\mu=0.$$

Next we show the above formula is still valid for any $\phi\in C(V)$, hence $-\Delta_pu+\overline{f}-f=0$, and $u$ is a solution to the equation $\Delta_pu=\overline{f}-f$. In fact, for any $f\in C(V)$, it is easy to see $\int_V\Delta_pfd\mu=0$ and $\int_V(\overline{f}-f)d\mu=0$. For any $\phi\in C(V)$, note $\phi-\overline{\phi}\in \Lambda$, then
$$\int_V(-\Delta_pu+\overline{f}-f)(\phi-\overline{\phi})d\mu=0.$$
Since
$$\int_V(-\Delta_pu+\overline{f}-f)\overline{\phi}d\mu=\overline{\phi}\int_V(\overline{f}-f)d\mu-\overline{\phi}\int_V\Delta_pud\mu=0,$$
then it follows $$\int_V(-\Delta_pu+\overline{f}-f)\phi d\mu=0.$$

\section{Proof of Theorem \ref{thm-c=0}}
\label{sect-prof-thm-c=0}
\subsection{Necessary condition}
Assume $u$ is a solution to the equation $\Delta_pu=-he^u$. Then
$$\sum\limits_{i\thicksim j}\omega_{ij}|u_j-u_i|^{p-2}(u_j-u_i)=-\mu_ih_ie^{u_i}.$$
Using the following formula
$$\sum\limits_{i\in V}\sum\limits_{i\thicksim j}a_{ij}=\sum\limits_{i\thicksim j}(a_{ij}+a_{ji})$$
for any function $a$ defined on the directed edges, we see
$$0=\sum\limits_{i\in V}\sum\limits_{i\thicksim j}\omega_{ij}|u_j-u_i|^{p-2}(u_j-u_i)=-\sum\limits_{i\in V}\mu_ih_ie^{u_i},$$
which implies that $h$ changes sign. Moreover,
\begin{equation*}
\begin{aligned}
-\int_Vhd\mu=-\sum\limits_{i\in V}\mu_ih_i=&\sum\limits_{i\in V}e^{-u_i}\sum\limits_{i\thicksim j}\omega_{ij}|u_j-u_i|^{p-2}(u_j-u_i)\\
=&\sum\limits_{i\thicksim j}\omega_{ij}|u_j-u_i|^{p-2}(u_j-u_i)(e^{-u_i}-e^{-u_j})\geq 0.
\end{aligned}
\end{equation*}
Hence $\int_Vhd\mu\leq0$. If $\int_Vhd\mu=0$, then $u_j=u_i$ for each edge $i\thicksim j$. This implies that $u$ equals to a constant since the graph is connected. Then $0=\Delta_pu=-he^u$, which contradicts with $h\not\equiv0$.
\subsection{Sufficient condition}
Now assume $h$ changes sign and $\overline{h}<0$. Set
$$\mathcal{B}_1=\left\{f\in W^{1,\,p}(G):\int_Vfd\mu=0, \;\int_Vhe^fd\mu=0\right\}.$$
Consider the $p$-th Dirichlet energy functional
$$D(f)=\int_E|\nabla f|^pd\omega,$$
and let $$a=\inf\limits_{f\in\mathcal{B}_1}D(f)$$
Since $\mathcal{B}_1$ is not empty (this is proved in section 4 of \cite{GLY} by Grigor'yan-Lin-Yang), $a\in\mathds{R}$ is well defined. Choose $f_n\in\mathcal{B}_1$, such that $D(f_n)\rightarrow a$ as $n\rightarrow\infty$. By Lemma \ref{lem-p-poincare}, $f_n$ is bounded in $W^{1,\,p}(G)$. Then by Lemma \ref{lem-Sobolev-embedding}, there exists some $\hat{f}\in C(V)$ such that up to a subsequence, $f_n\rightarrow \hat{f}$
in $W^{1,\,p}(G)$. We may well denote this subsequence as $f_n$. It's easy to see $\hat{f}\in\mathcal{B}_1$ and $D(\hat{f})=a$. Based on the method of Lagrange multipliers, we calculate the Euler-Lagrange equation of $D(f)$ under constriant conditions as in $\mathcal{B}_1$ as follows. For any $\phi\in W^{1,\,p}(G)$, there holds
\begin{equation*}
\begin{aligned}
0=&\frac{d}{dt}\Big|_{t=0}\Big(D(\hat{f}+t\phi)-\lambda\int_Vhe^{\hat{f}+t\phi}d\mu-\gamma\int_V(\hat{f}+t\phi)d\mu\Big)\\
=&-p\int_V\Big(\Delta_p\hat{f}+\frac{\lambda}{p}he^{\hat{f}}+\frac{\gamma}{p}\Big)\phi d\mu.
\end{aligned}
\end{equation*}
Hence we get
\begin{equation}\label{euq-Euler-Lagr-c=0}
\Delta_p\hat{f}=-\frac{\lambda}{p}he^{\hat{f}}-\frac{\gamma}{p}.
\end{equation}
Integrating the above equation, we get $\gamma=0$. Next we show $\lambda>0$. By the equation (\ref{euq-Euler-Lagr-c=0}), we have
$\Delta_p\hat{f}_i=-\frac{\lambda}{p}he^{\hat{f}_i}$ and then
\begin{equation*}
\begin{aligned}
-\frac{\lambda}{p}\int_Vhd\mu=&-\frac{\lambda}{p}\sum\limits_{i\in V}\mu_ih_i\\
=&\sum\limits_{i\in V}e^{-\hat{f}_i}\sum\limits_{i\thicksim j}\omega_{ij}|\hat{f}_j-\hat{f}_i|^{p-2}(\hat{f}_j-\hat{f}_i)\\
=&\sum\limits_{i\thicksim j}\omega_{ij}|\hat{f}_j-\hat{f}_i|^{p-2}(\hat{f}_j-\hat{f}_i)(e^{-\hat{f}_i}-e^{-\hat{f}_j})\\
\geq& 0.
\end{aligned}
\end{equation*}
This implies $\lambda\geq0$. If $\lambda=0$, then $\Delta_p\hat{f}=0$ and hence $\hat{f}$ is a constant by Lemma \ref{lem-Liouville}. Note we have
already proved $\hat{f}\in\mathcal{B}_1$. From $\int_V\hat{f}d\mu=0$, we know $\hat{f}=0$. From $\int_Vhe^{\hat{f}}d\mu=0$, we further know $h\equiv0$. This contradicts with $h\not\equiv0$. Hence $\lambda>0$. Now set $$u=\hat{f}+\ln \frac{\lambda}{p},$$
then $u$ satisfies
$\Delta_pu=-he^{u}.$

\section{Proof of Theorem \ref{thm-c>0}}
\label{sect-prof-thm-c>0}
\subsection{Necessary condition}
Assume $u$ is a solution to the equation $\Delta_pu=c-he^u$ with $c>0$. Integrating this equation, we get
\begin{equation*}
\int_Vhe^u=c \mathrm{Vol}(G)>0.
\end{equation*}
Hence $h$ must be positive somewhere on $V$.
\subsection{Sufficient condition}
Now assume $h$ is positive somewhere on $V$. Set
$$\mathcal{B}_2=\left\{f\in W^{1,\,p}(G):\int_Vhe^fd\mu=c\mathrm{Vol}(G)\right\}.$$
$\mathcal{B}_2$ is not empty (this is proved in section 5 of \cite{GLY} by Grigor'yan-Lin-Yang). Consider the following energy functional
$$I(f)=\frac{1}{p}D(f)+c\int_Vfd\mu,\;f\in\mathcal{B}_2.$$
We shall prove $I(f)\geq-C_{p,h,G}$ for all $f\in\mathcal{B}_2$. For any $f\in\mathcal{B}_2$, by
$$\int_Vhe^{f-\overline{f}}d\mu=e^{-\overline{f}}\int_Vhe^{f}d\mu=e^{-\overline{f}}c\mathrm{Vol}(G)>0,$$
we can rewrite $I(f)$ as
$$I(f)=\frac{1}{p}D(f)-c\mathrm{Vol}(G)\ln\int_Vhe^{f-\overline{f}}d\mu+c\mathrm{Vol}(G)\ln(c\mathrm{Vol}(G)).$$
Set $g=(f-\overline{f})\big(D(f-\overline{f})\big)^{-\frac{1}{p}}=(f-\overline{f})D(f)^{-\frac{1}{p}}$, then $\int_Vgd\mu=0$ and $D(g)=1$. Then by
Lemma \ref{lem-Trudinger-Moser}, we get
$$\int_Ve^{\beta|g|^{\alpha}}d\mu\leq C_{\alpha,\beta,p,G}.$$
Using an elementary inequality
$$ab\leq\frac{\epsilon^p}{p}a^p+\frac{\epsilon^{-q}}{q}b^q,$$
we have for any $\epsilon>0$,
\begin{equation*}
\begin{aligned}
\int_Vhe^{f-\overline{f}}d\mu\leq\int_V|h|_Me^{|g|D(f)^{-\frac{1}{p}}}d\mu
\leq&|h|_M\int_Ve^{\frac{\epsilon^p}{p}D(f)+\frac{\epsilon^{-q}}{q}|g|^q}d\mu\\
=&|h|_Me^{\frac{\epsilon^p}{p}D(f)}\int_Ve^{\frac{\epsilon^{-q}}{q}|g|^q}d\mu\\
\leq&C_{\epsilon,p,h,G}e^{\frac{\epsilon^p}{p}D(f)}.
\end{aligned}
\end{equation*}
Thus
\begin{equation*}
\begin{aligned}
I(f)\geq&\frac{1}{p}D(f)-c\mathrm{Vol}(G)\ln\big(C_{\epsilon,p,h,G}e^{\frac{\epsilon^p}{p}D(f)}\big)+c\mathrm{Vol}(G)\ln(c\mathrm{Vol}(G))\\
\geq&\frac{1}{p}D(f)\big(1-c\mathrm{Vol}(G)\epsilon^p\big)-C_{\epsilon,p,h,G}.
\end{aligned}
\end{equation*}
Choosing $\epsilon=\big(2c\mathrm{Vol}(G)\big)^{-\frac{1}{p}}$, we obtain for all $f\in\mathcal{B}_2$,
\begin{equation}\label{equ-I}
I(f)\geq\frac{1}{2p}D(f)-C_{p,h,G}.
\end{equation}
Therefore,
$$b=\inf\limits_{f\in\mathcal{B}_2}I(f)\in \mathds{R}$$
is well defined. Choose $f_n\in\mathcal{B}_2$, such that $I(f_n)\rightarrow b$ as $n\rightarrow\infty$. We may well suppose
$$I(f_n)\leq b+1$$
for each $n$. Then by the estimate (\ref{equ-I}), we have
\begin{equation}\label{equ-D(f_n)}
D(f_n)\leq2p(b+1+C_{p,h,G})=C_{b,p,h,G}.
\end{equation}
Since $f_n\in\mathcal{B}_2$, we have
\begin{equation}\label{equ-|f_n|}
\begin{aligned}
\big|\overline{f_n}\big|=\mathrm{Vol}(G)\Big|\int_Vf_nd\mu\Big|=&\mathrm{Vol}(G)c^{-1}\Big|I(f_n)-\frac{1}{p}D(f_n)\Big|\\
\leq&(1+b)C_{p,h,G}=C_{b,p,h,G}.
\end{aligned}
\end{equation}
Since $\int_V(f_n-\overline{f_n})d\mu=0$, then by Lemma \ref{lem-p-poincare},
\begin{equation}\label{equ-(fn-bar-fn)}
\|f_n-\overline{f_n}\|_p^p\leq C_{p,G}\int_E\big|\nabla (f_n-\overline{f_n})\big|^pd\omega
=C_{p,G}\int_E\big|\nabla f_n\big|^pd\omega=C_{p,G}D(f_n).
\end{equation}
By the estimates (\ref{equ-|f_n|}) and (\ref{equ-(fn-bar-fn)}), we obtain
\begin{equation}\label{equ-(fn)_p}
\|f_n\|_p\leq\|f_n-\overline{f_n}\|_p+\|\overline{f_n}\|_p\leq C_{b,p,h,G}.
\end{equation}
Therefore $f_n$ is bounded in $W^{1,\,p}(G)$ by the estimates (\ref{equ-D(f_n)}) and (\ref{equ-(fn)_p}). By Lemma \ref{lem-Sobolev-embedding}, there exists some $\hat{f}\in C(V)$ such that up to a subsequence, $f_n\rightarrow \hat{f}$
in $W^{1,\,p}(G)$. We may well denote this subsequence as $f_n$. It's easy to see $\hat{f}\in\mathcal{B}_2$ and $I(\hat{f})=b$. We calculate the Euler-Lagrange equation of $I(f)$ under constraint conditions as in $\mathcal{B}_2$ as follows. For any $\phi\in W^{1,\,p}(G)$, there holds
\begin{equation*}
\begin{aligned}
0=&\frac{d}{dt}\Big|_{t=0}\Big\{I(\hat{f}+t\phi)-\lambda\Big(\int_Vhe^{\hat{f}+t\phi}d\mu-c\mathrm{Vol}(G)\Big)\Big\}\\
=&\int_V\big(-\Delta_p\hat{f}+c-\lambda he^{\hat{f}}\big)\phi d\mu.
\end{aligned}
\end{equation*}
Hence we get
\begin{equation}\label{euq-Euler-Lagr-c=0}
\Delta_p\hat{f}=c-\lambda he^{\hat{f}}.
\end{equation}
Integrating the above equation, we get $\lambda=1$. Then $\hat{f}$ is a solution to the equation (\ref{def-p-KW-eq}).

\section{Proof of Theorem \ref{thm-c<0}}
\label{sect-prof-thm-c<0}
\subsection{Necessary condition}
Assume $u$ is a solution to the equation $\Delta_pu=c-he^u$ with $c<0$. Then for each $i\in V$,
$$c\mu_i-\mu_ih_ie^{u_i}=\sum\limits_{j\thicksim i}\omega_{ij}|u_j-u_i|^{p-2}(u_j-u_i).$$
This leads to
\begin{equation*}
\begin{aligned}
c\sum\limits_{i\in V}\mu_ie^{-u_i}-\sum\limits_{i\in V}\mu_ih_i=
&\sum\limits_{i\in V}e^{-u_i}\sum\limits_{j\thicksim i}\omega_{ij}|u_j-u_i|^{p-2}(u_j-u_i)\\
=&\sum\limits_{j\thicksim i}\omega_{ij}|u_j-u_i|^{p-2}(u_j-u_i)(e^{-u_i}-e^{-u_j})\geq0.
\end{aligned}
\end{equation*}
Because $c\sum\limits_{i\in V}\mu_ie^{-u_i}<0$, so $\sum\limits_{i\in V}\mu_ih_i<0$, that is $\int_Vhd\mu<0$, and hence $\overline{h}<0$.
\subsection{Sufficient condition}
Now assume $\overline{h}<0$. We shall prove the Theorem \ref{thm-c>0} by using the method of upper and lower solutions. We call a function $u_-$ a lower solution to the equation (\ref{def-p-KW-eq}), if for all $i\in V$ there holds
$$\Delta_pu_--c+he^{u_-}\geq0.$$
Similarly,  $u_+$ is called an upper solution to the equation (\ref{def-p-KW-eq}), if for all $i\in V$ there holds
$$\Delta_pu_+-c+he^{u_+}\leq0.$$
\begin{lemma}\label{lem-exist-low-solution}
Consider the equation (\ref{def-p-KW-eq}) with $c<0$ and $\overline{h}<0$. Then for any given function $f\in C(V)$, there exists a lower solution $u_-$ to (\ref{def-p-KW-eq}) with $u_-\leq f$.
\end{lemma}
\begin{proof}
Note $\Delta_p(-A)-c+he^{-A}=-c+he^{-A}\geq0$ if the constant $A>0$ is big enough, hence $u_-=-A$ is a lower solution, which can be chosen smaller than any $f$.\qed
\end{proof}

\begin{lemma}\label{lem-exist-solu-c<0}
Consider the equation (\ref{def-p-KW-eq}) with $c<0$ and $\overline{h}<0$. Assume (\ref{def-p-KW-eq}) has a lower solution $u_-$ and an upper solution $u_+$, moreover, $u_-\leq u_+$. Then (\ref{def-p-KW-eq}) has a solution $u$ satisfying $u_-\leq u\leq u_+$.
\end{lemma}
\begin{proof}
Set $k_1=\max\{1,-h\}$ and $k=k_1e^{u_+}>0$. Set $u_0=u_+$, and for each $n\geq0$,
$$u_{n+1}=L^{-1}(c-he^{u_n}-ku_n).$$
By Theorem \ref{thm-L}, all $u_n\in C(V)$, $n\geq0$ are well defined. We claim that
$$u_-\leq u_{n+1}\leq u_n\leq\cdots\leq u_1\leq u_0=u_+.$$
To prove the above claim, we just need to prove
$$Lu_-\geq Lu_{n+1}\geq Lu_n\geq\cdots\geq Lu_1\geq Lu_0=Lu_+$$
by the ``preserve order" property of $L$ in Theorem \ref{thm-L}. We prove this by method of induction. It's easy to see $Lu_1-Lu_0=c-he^{u_+}-\Delta_pu_+\geq0$. By the mean value theorem,
\begin{equation*}
\begin{aligned}
Lu_--Lu_1=&(\Delta_pu_--ku_-)-(c-he^{u_+}-ku_+)\\
\geq&k(u_+-u_-)+h(e^{u_+}-e^{u_-})\\
=&k_1e^{u_+}(u_+-u_-)+he^\xi(u_+-u_-)\\
\geq&k_1(e^{u_+}-e^\xi)(u_+-u_-)\geq0,
\end{aligned}
\end{equation*}
where $u_-\leq\xi\leq u_+$. Hence we obtain $Lu_-\geq Lu_1\geq Lu_0=Lu_+$. Suppose
$$Lu_-\geq Lu_n\geq Lu_{n-1}\geq Lu_+$$
for some $n\geq1$. Then
\begin{equation*}
\begin{aligned}
Lu_{n+1}-Lu_n=&k_1e^{u_+}(u_{n-1}-u_n)+he^\xi(u_{n-1}-u_n)\\
\geq&k_1(e^{u_+}-e^\xi)(u_{n-1}-u_n)\geq0,
\end{aligned}
\end{equation*}
where $u_n\leq\xi\leq u_{n-1}\leq u_+$. Similarly,
\begin{equation*}
\begin{aligned}
Lu_{-}-Lu_{n+1}=&(\Delta_pu_--ku_-)-(c-he^{u_n}-ku_n)\\
\geq&k(u_n-u_-)+h(e^{u_n}-e^{u_-})\\
=&(k_1e^{u_+}+he^\xi)(u_n-u_-)\\
\geq&k_1(e^{u_+}-e^\xi)(u_n-u_-)\geq0,
\end{aligned}
\end{equation*}
where $u_-\leq\xi\leq u_{n}\leq u_+$. Hence we obtain $Lu_-\geq Lu_{n+1}\geq Lu_{n}\geq Lu_+$. By induction, we prove the claim above, which implies that there is a $u\in C(V)$, such that $u_n\rightarrow u$ as $n\rightarrow\infty$. Note
$$\Delta_pu_{n+1}-ku_{n+1}=Lu_{n+1}=c-he^{u_n}-ku_n,$$
then let $n\rightarrow\infty$, we get
$$\Delta_pu=c-he^{u}.$$
hence $u$ is a solution to the equation (\ref{def-p-KW-eq}).
\qed
\end{proof}

\begin{lemma}\label{lem-exist-upper-solution}
Consider the equation (\ref{def-p-KW-eq}) with $c<0$ and $\overline{h}<0$. Then there is a constant $C_h>0$, such that for any $c$ with $-C_h\leq c<0$, (\ref{def-p-KW-eq}) has an upper solution $u_+$.
\end{lemma}
\begin{proof}
By \ref{thm-bar-f-equation}, we can choose $v\in C(V)$, which is a solution to the equation $\Delta_pv=\overline{h}-h$.
Choose a constant $a>0$, such that
$$|e^{av}-1|\leq\frac{-\overline{h}}{2|h|_M}=\frac{-\overline{h}}{2\max\limits_{i\in V}|h_i|}.$$
Set $u_+=av+(p-1)\ln a$ and $C_h=\frac{-\overline{h}}{2}a^{p-1}$, then for any $c$ with $-C_h\leq c<0$, we obtain
\begin{equation*}
\begin{aligned}
\Delta_pu_+-c+he^{u_+}\leq&a^{p-1}h(e^{av}-1)+\frac{\overline{h}}{2}a^{p-1}\\
\leq&a^{p-1}|h|(e^{av}-1)+\frac{\overline{h}}{2}a^{p-1}\\
\leq&a^{p-1}|h|_M\frac{-\overline{h}}{2|h|_M}+\frac{\overline{h}}{2}a^{p-1}=0.
\end{aligned}
\end{equation*}\qed
\end{proof}

Now we define
\begin{equation}
c_-(h)=\inf\big\{c<0:\Delta_pu-c+he^u=0 \;\mathrm{has\;an\;upper\;solution}\big\}.
\end{equation}
By Lemma \ref{lem-exist-upper-solution}, $-\infty\leq c_-(h)<0$. For any $c$ with $c<c_-(h)$, (\ref{def-p-KW-eq}) has no upper solutions by the definition of $c_-(h)$. Since every solution is also an upper solution, (\ref{def-p-KW-eq}) has no solutions either for this $c$. For any $c$ with $c_-(h)<c<0$, it's easy to see that (\ref{def-p-KW-eq}) has at least an upper solution $u_+$. By Lemma \ref{lem-exist-low-solution}, we can find a lower solution $u_-$ to (\ref{def-p-KW-eq}) with $u_-\leq u_+$. Then by Lemma \ref{lem-exist-solu-c<0}, the equation (\ref{def-p-KW-eq}) has a solution for this $c$. Thus we finish the proof of Theorem \ref{thm-c<0}.

\section{Proof of Theorem \ref{thm-c<0-h<0}}
\label{sect-prof-thm-c<0-h<0}
Consider the equation (\ref{def-p-KW-eq}) with $c<0$, $h\leq0$ and $h\not\equiv0$. We need to find a solution to (\ref{def-p-KW-eq}).
By \ref{thm-bar-f-equation}, we can choose $v\in C(V)$, which is a solution to the equation $\Delta_pv=\overline{h}-h$.
Since $\overline{h}<0$ under the assumption of $h$, we can choose $a>0$, such that $a^{p-1}\overline{h}<c$. Fix any $b\in \mathds{R}$ with
$b>(p-1)\ln a-av$, and set $u_+=av+b$, then $e^{u_+}>a^{p-1}$. Hence
\begin{equation*}
\begin{aligned}
\Delta_pu_+-c+he^{u_+}=&a^{p-1}\overline{h}-c+he^{u_+}-ha^{p-1}\\
\leq&h(e^{u_{+}}-a^{p-1})\leq0,
\end{aligned}
\end{equation*}
which shows $u_{+}$ is an upper solution to (\ref{def-p-KW-eq}). By Lemma \ref{lem-exist-low-solution}, we can find a lower solution $u_-$ to (\ref{def-p-KW-eq}) with $u_-\leq u_+$. Then by Lemma \ref{lem-exist-solu-c<0}, the equation (\ref{def-p-KW-eq}) has a solution.\\

\noindent \textbf{Acknowledgements:} The author would like to thank Professor Gang Tian and Yanxun Chang for constant encouragement. The author would also like to thank Dr. Wenshuai Jiang, Xu Xu for many helpful conversations. The research is supported by National Natural Science Foundation of China under Grant No.11501027, and Fundamental Research Funds for the Central Universities (Nos. 2015JBM103, 2014RC028, 2016JBM071 and 2016JBZ012).

Huabin Ge: hbge@bjtu.edu.cn

Department of Mathematics, Beijing Jiaotong University, Beijing 100044, P.R. China
\end{document}